\documentclass[11pt,twoside]{amsart}
\usepackage{bbm}
\usepackage{bbm}
\usepackage{mathrsfs}
\usepackage{amsmath}
\usepackage{amsthm}
\usepackage{young}
\usepackage[enableskew]{youngtab}
\usepackage{amsfonts}
\usepackage{amssymb}
\usepackage{latexsym}
\usepackage[all]{xy}

\date{}
\allowdisplaybreaks[4] \footskip=15pt
\renewcommand{\uppercasenonmath}[1]{}

\makeatletter
   \DeclareMathSizes{\@xipt}{\@xipt}{7}{2}
\makeatother
\numberwithin{equation}{section} \theoremstyle{plain}
\newtheorem*{theorem*}{Main Theorem}
\newtheorem{theorem}{Theorem}[section]

\newtheorem*{corollary*}{Corollary}
\newtheorem{lemma}[theorem]{Lemma}
\newtheorem*{lemma*}{Lemma}
\newtheorem{proposition}[theorem]{Proposition}
\newtheorem*{proposition*}{Proposition}
\newtheorem{remark}[theorem]{Remark}
\newtheorem*{remark*}{Remark}

\newtheorem*{example*}{Example}
\newtheorem{definition}[theorem]{Definition}
\newtheorem*{definition*}{Definition}

\newtheorem*{conjecture*}{Conjecture}
\newtheorem*{ack*}{ACKNOWLEDGEMENTS}

\newcommand{\bc}{\begin{center}}
\newcommand{\ec}{\end{center}}

\newcommand{\pf}{\noindent\begin {proof}}
\newcommand{\epf}{\end{proof}}

\begin{document}
\newcommand{\Ho}{H}\newcommand{\Tom}{T}
\newcommand{\Eo}[1]{E^{#1}}
\newcommand{\To}[1]{T^{#1}}
\begin{center}
{\large  \bf A Basis of the $q$-Schur Module $\mathcal{A}^\lambda$ }

\vspace{0.8cm} {\small \bf  Xingyu Dai$^{a,\;(1)}$, \ \  Fang Li$^{a,\;(2)}$, \ \ Kefeng Liu$^{a,b\;(3)}$}\\
\vspace{0.1cm}
$^a$Center of Mathematical Sciences, Zhejiang University, Zhejiang 310027, China\\
$^b$Department of Mathematics, University of California, Los Angeles, USA\\
E-mail: $(1)$ daixingyu12@126.com;  $(2)$ fangli@zju.edu.cn;  $(3)$ liu@math.ucla.edu

\end{center}

\bigskip
\centerline { \bf  Abstract}

In this paper, we construct the so-called $q$-Schur modules as left principle ideals of cyclotomic $q$-Schur algebras, and prove that they are isomorphic to those cell modules defined in \cite{8} and \cite{15} in any level $r$. After that, mainly, we prove that these $q$-Schur modules are free and construct  their basis. This result gives the new versions of some several known results such as standard basis and the branching theorem. With the help of this realizations and the new basis, we give a new proof of the Branch rule of Weyl modules which was first discovered by Wada in \cite{23}.

\leftskip10truemm \rightskip10truemm
 \noindent
\\
\vbox to 0.3cm{}

 \leftskip0truemm
\rightskip0truemm
\bigskip

\section{\bf Introduction}

Weyl modules for a cyclotomic $q$-Schur algebra $\mathscr{S}_{n,r}$ have been investigated recently in the context of cellular algebras (see \cite{8}). These modules are defined as quotient modules of certain {\em permutation} modules, that is,  as {\em cell modules} via cellular basis.

However, the classical theory \cite{3} and the works \cite{10} \cite{14} in the case when $m=1,2$ suggested that a construction as \emph{submodules} without using cellular basis should exist in the case of Iwahori-Hecke algebra.
Following Dipper and James' work \cite{6},  when the level $l$ equals to one, {\em basis} and {\em structure} appearing in Hecke algebras can still be constructed in $q$-Schur algebras with a totally different way.

 This phenomena has a great change to stay valid in the case of cyclotomic $q$-Schur algebras with large level, which is the inspiration of this paper. We can solve the difficulties by constructing a series of principle left ideals of the cyclotomic $q$-Schur algebras, where each single one is generated by a single element $z_\lambda$. The element $z_\lambda$ we construct is $\varphi_{\lambda w}^1\cdot T_{w_\lambda}\cdot y_{\lambda'}$ by the right Ariki-Koike algebra $\mathcal{H}_{n,r}$-module structure, where the element $y_{\lambda'}$ and morphism $\varphi_{\lambda w}^d$ are defined in \ref{m>1} and \ref{mainbasis} respectively. i.e., {\em $q$-Schur module} $\mathcal{A}^\lambda$ is defined as $\mathscr{S}_{n,r}\cdot \varphi_{\lambda w}^1T_{w_\lambda}y_{\lambda'}$ (Definition \ref{mainbasis}). Then in Theorem 3.1, we prove that the $\mathcal{A}^\mu$ as $\mathscr{S}_{n,r}\cdot z_\mu$ is exactly a realization of Weyl modules in the category of modules over cyclotomic $q$-Schur algebras which is a generalization of Dipper and James' work \cite{6}. After that, we construct and prove a $R$-linear basis of the $q$-Schur module $\mathcal{A}^\mu$  in the main result as follows:  \\
\\
{\bf Theorem 3.5.}~~{\em
Suppose that $\lambda\in\Lambda_{n,r}^+(\textbf{m})$. Then the $q$-Schur module
$\mathcal {A}^\lambda$ is free as a $R$-module and
$\{\varphi^{1_A}_{\mu\lambda}\cdot z_\lambda|A\in \mathcal
{T}_\mu^{ss}(\lambda)\  and\
\mu\in\Lambda_{n,r}(\textbf{m})\}\subseteq\mathcal {A}^\lambda$ is a basis.}\\
\\
Here $\mu$ is any multipartition (defined in Section 2.1) and $A$ is  its semi-standard tableau (defined in Remark \ref{redefinition}). This theorem is something like ``the half way" of the semi-standard basis that appeared in \cite{8}. With the help of this basis constructed, we can show a new version of the Branch rule which happens in the category of modules over a cyclotomic $q$-Schur algebra.

The paper is organised as follows. In Section 3, we construct some left ideals $\{\mathcal{A}^\mu\}$, which are called \emph{$q$-Schur modules} over the cyclotomic $q$-Schur algebra ${}_R\mathscr{S}_{n,r}$, and prove that this $q$-Schur modules are the same as Weyl modules in \cite{8}. After that, we clarified that these ideals are spanned by the natural basis as $\{\varphi_{\mu\lambda}^{1_A}\cdot z_\lambda|\mu\in \Lambda_{n,r}(\textbf{m})\text{ and }$A$\in \mathcal{T}^{ss}_\mu(\lambda)\}$, just as a parallel work of Dipper and James in \cite{10}.
In Section 4,  using of these new basis in $q$-Schur modules, we construct their filtrations, as a new point of view to the Branch rule in Wada's work \cite{23}.

\section{\bf Prelimilaries}

\subsection{Some notations about tableaux}
A {\emph{composition}} $\lambda$ of $n$ is a finite sequence of
non-negative integers $(\lambda_1,\lambda_2,\ldots,\lambda_m)$ such
that $|\lambda|=\sum_i\lambda_i=n$. There is a partial
order $\unlhd$(resp. $\unrhd$) within compositions of $n$ as: we
denote $\lambda\unlhd\mu$ when
$\sum_{i=1}^k\lambda_i\leq\sum_{i=1}^k\mu_i$(resp.
$\sum_{i=1}^k\lambda_i\geq\sum_{i=1}^k\mu_i$) for all $1\leq k\leq
m$. Moreover, if a composition $\lambda$ satisfies that $\lambda_1\geq\lambda_2\geq\cdots\lambda_m$,  it is called a \emph{partition}.
For later use, let $\Lambda(n)$ (resp. $\Lambda^+(n)$) denote the
set of all compositions (resp. all partitions) of $r$.

Let $\mathfrak{S}_n$ denote the symmetric
group of all permutations of $1,\ldots,n$ with Coxeter generators
$s_i:=(i,i+1)$, and $\mathfrak{S}_\lambda$ the Young
subgroup corresponding to the composition $\lambda$ of $n$, which is
denoted by:
\begin{eqnarray*}
\mathfrak{S}_\lambda=\mathfrak{S}_\textbf{a}=\mathfrak{S}_{\{1,\ldots,a_1\}}\times\mathfrak{S}_{\{a_1+1,\ldots,a_2\}}\times\cdots\times\mathfrak{S}_{\{a_{n-1}+1,\ldots,a_n\}},
\end{eqnarray*}
where $\textbf{a}=[a_0,a_1,\ldots,a_n]$ with $a_0=0$ and
$a_i=\lambda_1+\cdots+\lambda_i$ for all $i=1,\ldots,m$. We
denote by $\mathscr{D}_\lambda$ the set of distinguished
representatives of right $\mathfrak{S}_\lambda$-cosets and write
$\mathscr{D}_{\lambda\mu}:=\mathscr{D}_\lambda\cap\mathscr{D}_\mu^{-1}$,
which is the set of distinguished representatives of double cosets
$\mathfrak{S}_\lambda\setminus\mathfrak{S}_n/\mathfrak{S}_\mu$.

One can identify a composition $\lambda$ with {\em{Young
diagram}} and we say that $\lambda$ is the \emph{shape} of the
corresponding Young diagram. A $\lambda$-\emph{tableau} is a
filling of the $n$ boxes of the Young diagram of $\lambda$ of the
numbers $1,2,\ldots,n$. Denote the set of $\lambda$-tableaux by
$\mathcal {T}(\lambda)$ and usually denote $\mathfrak{t}$ as an
element of $\mathcal {T}(\lambda)$.

 For
$\lambda\in\Lambda(n)$, let $\lambda'$ be the dual partition of
$\lambda$, i.e., $\lambda'_i:=\#\{j;\lambda_j\geq i\}$. There is a
unique element $w_\lambda\in \mathfrak{S}_n$ with the
\emph{trivial intersection property} in (4.1) of \cite{10}:
\begin{eqnarray}
w_\lambda^{-1}\mathfrak{S}_\lambda
w_\lambda\cap\mathfrak{S}_{\lambda'}=\{1\}.
\end{eqnarray}

We can represent $w_\lambda$ with help of Young diagrams. For
example, $\tiny\yng(3,2)$ represents $\lambda=(3,2)$, then $w_\lambda\in\mathfrak{S}_n$ is
defined by the equation $\mathfrak{t}^\lambda w_\lambda=\mathfrak{t}_\lambda$, where
$\mathfrak{t}^\lambda$ (resp. $\mathfrak{t}_\lambda$) is the
$\lambda$-tableau obtained by putting the number $1,2,\ldots,n$ in
order into the boxes from left to right down successive rows (resp.
columns). In the example, $$\mathfrak{t}^{(3,2)}=\tiny\young(123,45),\;\;\;\;\;
 \mathfrak{t}_{(3,2)}=\tiny\young(135,24).$$

\begin{definition}\cite{6}\label{chi}
Suppose that $\mathfrak{t}_1$ is a $\lambda$-tableau and $\mathfrak{t}_2$ is a
$\mu$-tableau, where both $\lambda$, $\mu\in \Lambda^+(n)$. Let $\chi(\mathfrak{t}_1,\mathfrak{t}_2)$ be a $n$-by-$n$ matrix whose
entry in row $i$ and column $j$ is the cardinality of following set:
$$\{ \text{entries in the first }i\text{  rows of }\mathfrak{t}_1\}\cap\{\text{entries in the first }j\text{ columns of }\mathfrak{t}_2\}.$$

\end{definition}

\begin{remark}\cite{6}\label{threeproper}
 If $\mathfrak{t}_1$ and $\mathfrak{t}_1'$ are $\lambda$-tableaux and $\mathfrak{t}_2$ and $\mathfrak{t}_2'$
are $\mu$-tableaux for $\lambda$ and $\mu\in \Lambda^+(n)$, then write $\chi(\mathfrak{t}_1,\mathfrak{t}_2)\geq\chi(\mathfrak{t}_1',\mathfrak{t}_2')$ if
each entry in $\chi(\mathfrak{t}_1,\mathfrak{t}_2)$ is not small than
corresponding one in $\chi(\mathfrak{t}_1',\mathfrak{t}_2')$. Write
$\chi(\mathfrak{t}_1,\mathfrak{t}_2)>\chi(\mathfrak{t}_1',\mathfrak{t}_2')$ if, in addition,
$\chi(\mathfrak{t}_1,\mathfrak{t}_2)\neq\chi(\mathfrak{t}_1',\mathfrak{t}_2')$.

The following properties are immediate from the definitions.
\begin{eqnarray}
\chi(\mathfrak{t}_1w,\mathfrak{t}_2w)&=&\chi(\mathfrak{t}_1,\mathfrak{t}_2)\quad  for\  all\  w\in\mathfrak{S}_r.\\
\chi(\mathfrak{t}_1w,\mathfrak{t}_2)&=&\chi(\mathfrak{t}_1,\mathfrak{t}_2)\quad  if\  w\in\mathfrak{S}_\lambda.\\
\chi(\mathfrak{t}_1,\mathfrak{t}_2w)&=&\chi(\mathfrak{t}_1,\mathfrak{t}_2)\quad  if\  w\in\mathfrak{S}_{\mu'}.
\end{eqnarray}

\end{remark}

Let $\textbf{m}=(m_1,\cdots, m_r)\in \mathbb{Z}_{>0}^r$ be an $r$-tuple of positive integers. Define a subset of $r$-composition of $n$ as:
\begin{eqnarray*}
\Lambda_{n,r}(\textbf{m})=
\Bigg\{
\mu=(\mu^{(1)},\cdots,\mu^{(r)})
\left|
\begin{array}{c}
\mu^{(k)}=(\mu_1^{(k)},\cdots,\mu_{m_k}^{(k)})\in \mathbb{Z}^{m_k}_{\geq0}\\
\sum_{k=1}^r\sum_{i=1}^{m_k}\mu^{(k)}_i=n
\end{array}
\right\}.
\end{eqnarray*}

We denote by $|\mu^{(k)}|=\sum_{i=1}^{m_k}\mu_i^{(k)}$ (resp. $|\mu|=\sum_{k=1}^r|\mu^{(k)}|$) the size of $\mu^{(k)}$ (resp. the size of $\mu$). We define the map $\zeta: \Lambda_{n,r}(\textbf{m})\rightarrow \mathbb{Z}_{\geq0}^r$ by $\zeta(\mu)=(|\mu^{(1)}|,|\mu^{(2)}|,\cdots,|\mu^{(r)}|)$ for
$\mu\in \Lambda_{n,r}(\textbf{m})$. Put $\Lambda^+_{n,r}(\textbf{m})=\{\lambda\in \Lambda_{n,r}(\textbf{m})|\lambda_1^{(k)}\geq\lambda_2^{(2)}\geq\cdots\geq\lambda_{m_k}^{(k)}\text{ for any }k=1,\cdots,r\}$.

Let $\lambda':=(\lambda^{(r)}{}',\ldots,\lambda^{(1)}{}')$
denote the $m$-composition \emph{\emph{dual}} to $\lambda$. By
concatenating the components of $\lambda$, the resulting composition
of $r$ will be denoted by
$$\overline{\lambda}:=\lambda^{(1)}\vee\cdots\vee\lambda^{(r)}.$$

We can also identify $\lambda\in \Lambda_{n,r}(\textbf{m})$ with a series of Young diagrams. For example,
$\lambda=((31),(21),(2))$ is identified with
$$\big(\tiny\yng(3,1),\  \tiny\yng(2,1),\  \tiny\yng(2)\big).$$

Similarly, we can define two tableaux $\mathfrak{t}^\lambda$ and $\mathfrak{t}_\lambda$ in multi-composition case.
Let $\mathfrak{t}^{\lambda}$ (resp. $\mathfrak{t}_\lambda$) be the $\lambda$-tableau obtained by
setting the numbers $1,\ldots,r$ in order into the boxes down
successive rows (resp. columns) in the first (resp. last) diagram of $\lambda$, then in the
second (resp. second last) diagram and so on. Due to the example above, we have

$$\mathfrak{t}^\lambda=(\tiny\young(123,4),\  \tiny\young(56,7),\  \young(89)).$$
$$\mathfrak{t}_\lambda=(\tiny\young(689,7),\  \tiny\young(35,4),\  \young(12)).$$

Give
the element $w_\lambda\in\mathfrak{S}_n$ by $\mathfrak{t}^\lambda
w_\lambda=\mathfrak{t}_\lambda$ corresponding to a $r$-partition
$\lambda=(\lambda^{(1)},\ldots,\lambda^{(r)})$ of $n$. More precisely, if $\mathfrak{t}^i$
denote the $i$-th subtableau of
$\mathfrak{t}^\lambda$,
 then define $w_{(i)}$ by $\mathfrak{t}^iw_{(i)}=\mathfrak{t}_i$.

\subsection{Ariki-Koike algebras and cyclotomic $q$-Schur algebras}

Now    recall the notion of the cyclotomic $q$-Schur algebra $\mathscr{S}_{n,r}$ from \cite{8} and the presentations of $\mathscr{S}_{n,r}$ by generators
and fundamental relations given in \cite{22}.

Let $R$ be a commutative ring, and take parameters $q,Q_1,\cdots,Q_r\in R$ such that $q$ is invertible in $R$. The Ariki-Koike algebra $\mathcal{H}_{n,r}$ is the associative algebra with $1$ over $R$ generated by $T_0,T_1,\ldots,T_{n-1}$ with the following defining relations:
\begin{eqnarray*}
&&(T_0-Q_1)(T_0-Q_2)\cdots(T_0-Q_r)=0,\\
&&(T_i-q)(T_i+q^{-1})=0\qquad\qquad\qquad\qquad\qquad (1\leq i\leq n-1),\\
&&T_0T_1T_0T_1=T_1T_0T_1T_0,\\
&&T_iT_{i+1}T_i=T_{i+1}T_iT_{i+1} \qquad\qquad\qquad\qquad\qquad (1\leq i\leq n-2),\\
&&T_iT_j=T_jT_i  \qquad\qquad\qquad\qquad\qquad\qquad \qquad\   (|i-j|\geq 2).
\end{eqnarray*}

The subalgebra of $\mathcal{H}_{n,r}$ generated by $T_1,\cdots,T_{n-1}$ is isomorphic to the\emph{ Iwahori-Hecke algebra} $\mathcal{H}_n$ (sometimes we write it $\mathcal{H}(\mathfrak{S}_n)$) in \cite{7}.
For $w\in \mathfrak{S}_n$,  denote by $\ell(w)$ the length of $w$ and by $T_w$ the \emph{standard basis} of $\mathcal{H}_{n}$ corresponding to $w$.

For each $r$-composition
$\lambda=(\lambda^{(1)},\ldots,\lambda^{(r)})$, define
$[\lambda]:=[a_0,a_1,\ldots,a_r]$ such that $a_0:=0$ and
$a_i:=\sum_{j=1}^i|\lambda^{(j)}|$. In the case of Iwahori-Hecke algebras, we can define a element
$m_\lambda\in \mathcal {H}_n$ as $m_\lambda:=\sum\limits_{w\in
\mathfrak{S}_\lambda} T_w$ and $w_\lambda\in\mathfrak{S}_n$ is defined in the above subsection.

\begin{definition}\label{m>1}
Let $\mathcal{H}_{n,r}$ be a cyclotomic Hecke algebra with generators
$\{T_0,T_1,\ldots,T_{n-1}\}$, and elements $L_1=T_0$,
$L_i=q^{-1}T_{i-1}L_{i-1}T_{i-1}$ for $i=2,\cdots,n$, and put
$\pi_0=1$, $\pi_a(x)=\Pi_{j=1}^a(L_j-x)$ for any $x\in R$ and any
positive integer $a$. Following \cite{8},  we can construct a series of numbers as
$\textbf{a}=[\lambda]=[a_0,a_1,\ldots,a_r]$. Define that
\begin{eqnarray*}
u^+_\textbf{a}=\pi_{a_1}(Q_2)\cdots\pi_{a_{r-1}}(Q_r)\ \ and \ \
u^-_\textbf{a}=\pi_{a_1}(Q_{r-1})\cdots\pi_{a_{r-1}}(Q_1),
\end{eqnarray*}
and, for $\lambda\in\Lambda_{n,r}(\textbf{m})$,  define that
\begin{eqnarray*}
x_\lambda:=u_{[\lambda]}^+m_{\overline{\lambda}}=m_{\overline{\lambda}}u_{[\lambda]}^+\
and\
y_\lambda:=u_{[\lambda]}^-m_{\overline{\lambda}}=m_{\overline{\lambda}}u_{[\lambda]}^-.
\end{eqnarray*}

Define the right ideal as $M^\lambda:=x_\lambda\mathcal{H}_{n,r}$ which is always called \emph{permutation module}.

\end{definition}

The cyclotomic $q$-Schur algebra $\mathscr{S}_{n,r}$ associated to $\mathcal{H}_{n,r}$ is defined by
\begin{eqnarray*}
_R\mathscr{S}_{n,r}={}_R\mathscr{S}_{\Lambda_{n,r}}(\textbf{m})=\text{End}_{\mathcal{H}_{n,r}}\bigg(\bigoplus\limits_{\mu\in \Lambda_{n,r}(\textbf{m})}M^\mu\bigg).
\end{eqnarray*}

In order to describe a presentation of ${}_R\mathscr{S}_{n,r}$, we prepare some notations.

Put $m=\sum_{k=1}^rm_k$, and define a ``\emph{dominant order in multipartitions}". i.e., for $\lambda,\mu\in\Lambda_{n,r}(\textbf{m})$ and $1\leq l\leq r$, $1\leq j\leq m_l$,
 $$\lambda\unrhd\mu\;\; \Leftrightarrow\;\; \sum_{i=1}^{l-1}|\lambda^{(i)}|+\sum_{k=1}^{j}\lambda^{(l)}_k\geq\sum_{i=1}^{l-1}|\mu^{(i)}|+\sum_{k=1}^{j}\mu^{(l)}_k.$$

For $(i,k)\in \Gamma'(\textbf{m})$, we define the elements $E_{(i,k)}$, $F_{(i,k)}\in {}_R\mathscr{S}_{n,r}$ by
\begin{eqnarray*}
\qquad E_{(i,k)}(m_\mu\cdot h)=
\left\{
\begin{array}{ccc}
q^{-\mu_{i+1}^{(k)}+1}\bigg(\sum\limits_{x\in X_\mu^{\mu+\alpha_{(i,k)}}}q^{\ell(x)}T_x^*\bigg)h_{+(i,k)}^\mu m_\mu\cdot h &if&\  \mu+\alpha_{(i,k)}\in \Lambda_{n,r}(\textbf{m}),\\
0&if&\  \mu+\alpha_{(i,k)} \notin \Lambda_{n,r}(\textbf{m}),
\end{array}
\right.
\end{eqnarray*}

\begin{eqnarray*}
\qquad F_{(i,k)}(m_\mu\cdot h)=
\left\{
\begin{array}{ccc}
q^{-\mu_{i}^{(k)}+1}\bigg(\sum\limits_{y\in X_\mu^{\mu-\alpha_{(i,k)}}}q^{\ell(x)}T_y^*\bigg) m_\mu\cdot h &if&\mu-\alpha_{(i,k)}\in \Lambda_{n,r}(\textbf{m}),\\
0&if&\mu-\alpha_{(i,k)} \notin \Lambda_{n,r}(\textbf{m}),
\end{array}
\right.
\end{eqnarray*}
for any $\mu\in \Lambda_{n,r}(\textbf{m})$ and $h\in \mathcal{H}_{n,r}$, where $h_{+(i,k)}^\mu=\left\{
\begin{array}{cc}
1&(i\neq m_k),\\
L_{N+1}-Q_{k+1}&(i=m_k).
\end{array}
\right.$

For $\lambda\in \Lambda_{n,r}(\textbf{m})$, we define the element $1_\lambda\in {}_R\mathscr{S}_{n,r}$ by
$$1_\lambda(m_\mu\cdot h)=\delta_{\lambda\mu}m_\lambda\cdot h$$
for $\mu\in \Lambda_{n,r}(\textbf{m})$ and $h\in \mathcal{H}_{n,r}$. In addition, we see that $\{1_\lambda|\lambda\in \Lambda_{n,r}(\textbf{m})\}$
is a set of pairwise orthogonal idempotents, and then $1=\sum_{\lambda\in \Lambda_{n,r}(\textbf{m})}1_\lambda$.

\begin{definition}\label{mainbasis}

 For any $m$ and $\mu\in \Lambda_{n,r}(\textbf{m})$, we now define a left principle ideal of cyclotomic $q$-Schur algebra as in the case  $m=1$ in \cite{6}:

 $\mathcal{A}^\mu\triangleq\mathscr{S}_{n,r}\varphi^1_{\mu\omega}T_{w_\mu}y_{\mu'}$ with
$\varphi^1_{\mu\omega}\in \text{Hom}_{\mathcal{H}_{n,r}}(\mathcal {H}_{n,r},M^\mu)$
is defined as $\varphi^1_{\mu\omega}(h):=x_\mu h$ for any
$h\in\mathscr{H}_{n,r}$. Meanwhile, the element $T_{w_\mu}y_{\mu'}$ acts on
$\varphi^1_{\mu\omega}$ induced by the right $\mathscr{H}_{n,r}$-module structure
of $M^\mu$. From now on, the
module $\mathcal {A}^\mu$ is called a \emph{$q$-Schur module}, and
denote the element $\varphi^1_{\mu\omega}T_{w_\mu}y_{\mu'}\in
\mathscr{S}_{n,r}$ by $z_\mu$.
\end{definition}

Recall in  \cite{11} that the
set of all $[\lambda]$ forms a poset $\Lambda[m,r]$ (where $m=\sum_i a_i$) which has
the same set $\Lambda(m,r)$ as all compositions of $m$
with at most $r$ parts but with different order.
Partial ordering on $\Lambda[m,r]$ is given by $\preceq:$
$[a_i]\preceq [b_i]$ if $a_i\leq b_i$ for all $i=1,\ldots,r$, while
$\Lambda(m,r)$ has the usual dominance order $\unlhd$.

The following results will be useful in the sequel (see
(2.8), (3.1), (3.4) in \cite{11}).
\begin{lemma}\cite{11}\label{lemdu}\;
Let $\textbf{a}, \textbf{b}\in \Lambda[m,r]$, and note $\mathcal{H}(\mathfrak{S}_n)$ as the Iwahori-Hecke algebra associated with $\mathfrak{S}_n$.
\begin{eqnarray}
 &&u^+_{\textbf{a}}\mathcal{H}_{n,r}u^-_{\textbf{b}'}=0\text{
unless }\textbf{a}\preceq\textbf{b}.\\
 &&u^+_{\textbf{a}}\mathcal
{H}(\mathfrak{S}_n)u^-_{\textbf{a}'}=\mathcal
{H}(\mathfrak{S}_{\textbf{a}})v_{\textbf{a}},\qquad  \text{where }\
v_{\textbf{a}}=u^+_{\textbf{a}}T_{w_{\textbf{a}}}u^-_{\textbf{a}'}.\\
&&u^+_{\textbf{a}}\mathcal{H}_{n,r}u^-_{\textbf{a}'}=u^+_{\textbf{a}}\mathcal
{H}(\mathfrak{S}_n)u^-_{\textbf{a}'}.\\
&&v_{\textbf{a}}\mathcal{H}_{n,r}\text{ is  a free } R\text{-submodule with basis } \{v_{\textbf{a}}T_w|w\in\mathfrak{S}_n\}.
\end{eqnarray}
\end{lemma}

\begin{definition}\cite{12}~ \label{defofw}
For $\lambda\in\Lambda^+_{n,r}(\textbf{m})$ and $\mu\in\Lambda_{n,r}(\textbf{m})$, a
$\lambda$-tableau of type $\mu$ denoted as $T$ is said
to be {\bf semistandard} if the following hold:

(i)  the entries in each row of each component of $T^{(k)}$ of $T$
are
non-decreasing;

(ii)  the entries in each column of each component $T^{(k)}$ of $T$
are strictly increasing;

(iii) if $(a,b,c)\in \lambda$, and $T(a,b,c)=(i,s)$ then $s\geq
c$.\\
Let $\mathcal {T}_\mu^{ss}(\lambda)$ be the set of semistandard
$\lambda$-tableau of type $\mu$ and denote $\mathcal
{T}^{ss}_\Lambda(\lambda)=\bigcup\limits_{\mu\in\Lambda}\mathcal{T}^{ss}_\mu(\lambda)$.
\end{definition}

The set
\begin{eqnarray}
\{\Psi_{ST}|S,T\in \mathcal {T}^{ss}_\Lambda(\lambda),\  \lambda\in
\Lambda^+(n,r)\},
\end{eqnarray}
 which is called the \emph{semi-standard} basis of cyclotomic $q$-Schur algebras in \cite{8},
forms a cellular basis of $\mathscr {S}_{n,r}$ in the sense of \cite{13}. Let $\mathscr
{S}_{n,r}^{\rhd\lambda}$ be the two sides ideal of $\mathscr{S}_{n,r}$ spanned by all
$\Psi_{ST}$, where $S,T\in \mathcal{T}_\Lambda^{ss}(\mu)$ and $\mu\rhd\lambda$
(i.e., $\mu:=\text{shape}(S)=\text{shape}(T)\rhd\lambda$), where shape($T$) means the partition corresponding to tableaux $T$.

In particular, let $\lambda\in\Lambda^+(n,r)$ be a partition and
recall that $T^{\lambda}$ is the unique semistandard $\lambda$-tableau of type
$\lambda$ (see \cite{8} and \cite{7}). From the definition, one sees that $\Psi_{T^\lambda
T^\lambda}$ can restrict to the identity map on $M_\lambda$, and sometimes we denote it by
$\Psi_{\lambda}$ .

With above notations, we can define the ``cell module" as a
submodule of $\mathscr {S}_{n,r}/\mathscr {S}_{n,r}^{\rhd\lambda}$:
\begin{eqnarray}\label{weylm}
W^\lambda=\mathscr{S}_{n,r}\bar{\Psi}_\lambda, \qquad  where\
\bar{\Psi}_\lambda:=(\mathscr{S}_{n,r}^{\rhd\lambda}+\Psi_\lambda)/\mathscr{S}_{n,r}^{\rhd\lambda}.
\end{eqnarray}
The module $W^\lambda$ is called a {\em Weyl module} in \cite{8}.

\section{\bf Main theorem and its proof}

We now prove $q$-Schur module given above is isomorphic to those in \cite{8} as ``cell modules" when $\lambda\in
\Lambda^+_{n,r}(\textbf{m})$. Recall the definitions given in \ref{defofw}.

\begin{theorem}\label{mtsec2.2}
For each $\lambda\in\Lambda_{n,r}^+(\textbf{m})$, we have the following $\mathscr {S}_{n,r}$-module
isomorphism:
$$\mathcal {A}^\lambda\cong W^{\lambda}.$$
\end{theorem}

\begin{proof}
Consider the epimorphism:
\begin{eqnarray*}
\theta:\mathscr {S}_{n,r}\Psi_{\lambda}\longrightarrow \mathscr
{S}_{n,r}z_\lambda;\quad h\Psi_{\lambda}\mapsto
hz_\lambda=h\varphi_{\lambda\omega}^1T_{w_\lambda}y_{\lambda'}=h\varphi_{\bar{\lambda}\omega}^1\cdot
T_{w_{(1)}\cdots w_{(r)}}
y_{\mu^{(1)}{}'\vee\cdots\vee\mu^{(r)}{}'}\cdot v_{[\mu]}.
\end{eqnarray*}
Suppose that $T\in\mathcal {T}^{ss}_\lambda(\mu)$ and $S\in\mathcal
{T}^{ss}_\nu(\mu)$ with $\mu\in\Lambda_{n,r}(\textbf{m})$ and
$\nu\in\Lambda_{n,r}(\textbf{m})$. By the definition of $\Psi_{ST}$ in \cite{8}
and semistandard basis theorem \cite{8} (\textbf{6.6}), we trivially
find that the set $\{\Psi_{ST}|T\in\mathcal {T}_\lambda^{ss}(\mu),\
S\in\mathcal {T}_\nu^{ss}(\mu)\  with\
\mu\unrhd\lambda, \mu\in\Lambda_{n,r}^+(\textbf{m}), \nu\in\Lambda_{n,r}(\textbf{m})\}$ is
a $R$-basis of $\mathscr {S}_{n,r}\Psi_\lambda$. More precisely, we can
write this basis as
\begin{eqnarray}\label{vbasis}
 \{\Psi_{TT^{\lambda}}|T\in
\mathcal{T}_\nu^{ss}(\lambda)\}\cup\{\Psi_{ST}|T\in\mathcal{T}_\lambda^{ss}(\mu)\
 and\  S\in\mathcal{T}_\nu^{ss}(\mu)\  with\  \mu\rhd\lambda\}.
\end{eqnarray}

Then, obviously, we have that $$W^{\lambda}\cong\mathscr
{S}_{n,r}\Psi_\lambda/(\mathscr {S}_{n,r}\Psi_\lambda\cap\mathscr
{S}_{n,r}^{\rhd\lambda}).$$

We claim that, with $\mu\unrhd\lambda$ and $\lambda\in\Lambda_{n,r}^+(\textbf{m})$,
$\nu\in\Lambda_{n,r}(\textbf{m})$, if
  $\theta(\Psi_{ST})=\theta(\Psi_{ST}\Psi_{T^{\lambda}T^{\lambda}})
=\Psi_{ST}\varphi_{\lambda\omega}^1T_{w_\lambda}y_{\lambda'}\neq0$
,then $\mu=\lambda$.

 Consider the action on the unit of $\mathcal{H}_{n,r}$:
\begin{eqnarray*}
\Psi_{ST}\varphi_{\lambda\omega}^1T_{w_\lambda}y_{\lambda'}(1)&=&m_{ST}T_{w_\lambda}y_{\lambda'}\\
&=&\sum\limits_{\substack{\mathfrak{t}\in
\text{Std}(\mu)\\
\lambda(\mathfrak{t})=T}}m_{S\mathfrak{t}}T_{w_\lambda}y_{\lambda'}=\sum\limits_{\substack{\mathfrak{t}\in
\text{Std}(\mu)\\
\lambda(\mathfrak{t})=T}}\sum\limits_{\substack{\mathfrak{s}\in
\text{Std}(\mu)\\
\nu(\mathfrak{s})=S}}m_{\mathfrak{s}\mathfrak{t}}T_{w_\lambda}y_{\lambda'}\\
&=&\sum\limits_{\mathfrak{s},\mathfrak{t}}T_{d(\mathfrak{s})}x_\mu
T_{d(\mathfrak{t})}T_{w_\lambda}y_{\lambda'}=\sum\limits_{\mathfrak{s},\mathfrak{t}}T_{d(\mathfrak{s})}x_{\bar{\mu}}u^+_{[\mu]}
T_{d(\mathfrak{t})}T_{w_\lambda}u^-_{[\lambda']}y_{\bar{\lambda}'}\\&=&(*).
\end{eqnarray*}
Recall that by Lemma \ref{lemdu},
$u^+_{\textbf{a}}\mathcal{H}_{n,r}u^-_{\textbf{b}'}=0$ unless
$\textbf{a}\preceq\textbf{b}$.
$\Psi_{ST}\varphi_{\lambda\omega}^1T_{w_\lambda}y_{\lambda'}\neq0$
implies that for some $\mathfrak{s}$ and $\mathfrak{t}$ above, that
$T_{d(\mathfrak{s})}x_{\bar{\mu}}u^+_{[\mu]}
T_{d(\mathfrak{t})}T_{w_\lambda}u^-_{[\lambda']}y_{\bar{\lambda}'}\neq0$.
Thus, this condition shows that $[\mu]\preceq [\lambda]$. On the other hand, with
the assumption in above claim, i.e., $\mu\unrhd\lambda$, it is obviously that $[\mu]\succeq[\lambda]$ by the definition of $[\mu]$, $[\lambda]$ and $\unrhd$, $\succeq$ .
So $[\mu]=[\lambda]$. Then we find
\begin{eqnarray*}
(*)&=&\sum\limits_{\substack{\mathfrak{s},\mathfrak{t} \\
[\mu]=[\lambda]}}T_{d(\mathfrak{s})}x_{\bar{\mu}}u^+_{[\mu]}
T_{d(\mathfrak{t})}T_{w_\lambda}u^-_{[\mu]'}y_{\bar{\lambda}'}\\
&=&\sum\limits_{\substack{\mathfrak{s},\mathfrak{t} \\
[\mu]=[\lambda]\\
h'\in\mathfrak{S}_{[\mu]}}}T_{d(\mathfrak{s})}x_{\bar{\mu}}h'v_{[\mu]}y_{\bar{\lambda}'}\qquad\qquad \text{by (2.6) and (2.7) in Lemma \ref{lemdu}}\\
&=&\sum\limits_{\substack{\mathfrak{s},\mathfrak{t} \\
[\mu]=[\lambda]\\
h'_i\in\mathfrak{S}_{\{|\lambda_{i-1}|+1,\cdots,|\lambda_i|\}}}}T_{d(\mathfrak{s})}x_{\mu^{(1)}\vee\cdots\vee\mu^{(r)}}h'_1\cdots h'_my_{\lambda^{(1)}{}'\vee\cdots\vee\lambda^{(r)}{}'}v_{[\mu]}\qquad \qquad\text{by \cite{9}}\\
&=&\sum\limits_{\substack{\mathfrak{s},\mathfrak{t} \\
[\mu]=[\lambda]\\
h'_i\in\mathfrak{S}_{\{|\lambda_{i-1}|+1,\cdots,|\lambda_i|\}}}}T_{d(\mathfrak{s})}(x_{\mu^{(1)}}h'_1y_{\lambda^{(1)}{}'})\cdots (x_{\mu^{(r)}}h'_my_{\lambda^{(r)}{}'})v_{[\mu]}\quad .\\
\end{eqnarray*}
Since $[\lambda]=[\mu]$, the fact that this is non-zero implies, by \cite{10}
(\textbf{4.1}), $\lambda^{(i)}\unrhd\mu^{(i)}$ for all
$i=1,\ldots,r$. On the other hand, by \cite{9} (\textbf{1.6}),
$\mu\unrhd\lambda$ and $[\mu]=[\lambda]$ implies
$\mu^{(i)}\unrhd\lambda^{(i)}$, with $1\leq i\leq r$. Hence
$\mu^{(i)}=\lambda^{(i)}$ for all $i$, and therefore, $\mu=\lambda$.
This completes the proof of above claim.

By the claim and the display in (\ref{vbasis}), one see that
$$ker\theta=\{\Psi_{ST}~|~T\in\mathcal{T}_\lambda^{ss}(\mu) \text{ and}\  S\in\mathcal{T}_\nu^{ss}(\mu) \text{  with}\  \mu\rhd\lambda\}
=\mathscr {S}_{n,r}\Psi_\lambda\cap\mathscr {S}_{n,r}^{\rhd\lambda}.$$
Therefore, $\mathcal {A}^\lambda\cong W^{\lambda}$.
\end{proof}

\begin{definition}\cite{10}~~\label{pair}
For $w\in\mathfrak{S}_n$ and $S\in\mathcal {T}_\lambda(\mu)$ with
$\lambda,\mu\in\Lambda(n,r)$,  define a map
\begin{eqnarray}\label{biformcom}
\mathfrak{S}_n\times\mathcal
{T}_\lambda(\mu)\longrightarrow\mathscr{D}_\lambda
\qquad \qquad (w,S)\longmapsto w_S
\end{eqnarray}
where the element $w_S$ is defined by the row-standard
$\lambda$-tableau $\mathfrak{t}^\lambda w_S$ for which $i$ belongs
to the row $a$ if the place occupied by $i$ in $\mathfrak{t}^\mu w$ is
occupied by $a$.
\end{definition}

For example,
$S=\tiny\young(123,12)$ and $\mathfrak{t}^\mu w=\tiny\young(124,35)$
with $\mu=(3,2)$ and $\lambda=(2,2,1)$, then
$\mathfrak{t}^\lambda w_S=\tiny\young(13,25,4)$.

\begin{remark}\label{redefinition}
Let $\mathcal{T}^{ss}_\lambda(\mu)$ be the set of all
semi-standard $\mu$-tableaux of type $\lambda$, with
$\lambda\text{ and }\mu\in\Lambda_{n,r}(\textbf{m})$. For any
$S\in\mathcal{T}^{ss}_\lambda(\mu)$, we define $1_S:=1_{\bar{S}}$.
Since $S$ is a semi-standard $\mu$-tableau of type $\lambda$, it implies that $\bar{S}$ is
a row-standard $\bar{\mu}$-tableau of type $\bar{\lambda}$, as in
\cite{1}.

\end{remark}

We compare the definition of
semi-standard tableaux which appears in \cite{8} with that in \cite{1}. Note that every entry in $S$ is written as the
symbol $(i,j)$ and is replaced by $i+\sum_{k=1}^{j-1}m_k$, for $1\leq i\leq m_j$, $1\leq j\leq n$.

Then, by the definition above, we obtain the following
consequence:
\begin{lemma}\label{lemmas2}

Suppose that $u\in \mathfrak{S}_r$ and
$w\in\mathfrak{S}_{\mu^{(1)'}\vee\cdots\vee\mu^{(r)'}}$, with
$\lambda,\mu\in\Lambda_{n,r}(\textbf{m})$. Then
$\varphi^1_{\bar{\lambda}\omega}T_uT_w$ is a linear combination of
terms $\varphi^d_{\bar{\lambda}\omega}$
($d\in\mathscr{D}_{\bar{\lambda}}$) for which
$\chi(t^{\bar{\lambda}}d,t^{\bar{\mu}}w_{(1)}\cdots
w_{(r)})=\chi(t^{\bar{\lambda}}u,t^{\bar{\mu}}w_{(1)}\cdots
w_{(r)})$.
\end{lemma}

\begin{proof}
 The conclusion is ture when $w=1$ since
$\varphi_{\bar{\lambda}\omega}^1T_u=\varphi_{\bar{\lambda}\omega}^u$
for some $u\in\mathfrak{S}_n$.
Below we assume that $w\neq 1$.

For
some $w'\in\mathfrak{S}_n$ and some
$a=(i,i+1)\in\mathfrak{S}_{\mu^{(1)'}\vee\cdots\vee\mu^{(r)'}}$, we
have that $w=w'a$, and without lose generality, we can set
$(i,i+1)\in \mathfrak{S}_{\mu^{(1)'}}$ satisfying:
\begin{eqnarray*}
 w'=w'_1\cdots
w'_r,\   w=w_1\cdots w_r\quad  \text{with}\quad  w_1'(i,i+1)&=&w_1,\\
 w_i&=&w_i'\quad
for\ i=2,\cdots,r.
\end{eqnarray*}

By induction on length $\ell(w)$, we have $\varphi^1_{\bar{\lambda}\omega}T_uT_{w'}$ as
a linear combination of terms $\varphi_{\bar{\lambda}\omega}^d$
($d\in\mathscr{D}_{\bar{\lambda}}$) for which
$\chi(t^{\bar{\lambda}}d,t^{\bar{\mu}}w_{(1)}\cdots
w_{(r)})=\chi(t^{\bar{\lambda}}u,t^{\bar{\mu}}w_{(1)}\cdots
w_{(r)})$.

Consider
$$\varphi_{\bar{\lambda}\omega}^1T_uT_w=\varphi_{\bar{\lambda}\omega}^1T_uT_wT_a=\sum\limits_{\chi(t^{\bar{\lambda}}d,t^{\bar{\mu}}w_{(1)}\cdots
w_{(r)})=\chi(t^{\bar{\lambda}}u,t^{\bar{\mu}}w_{(1)}\cdots
w_{(r)})}C_d\varphi_{\bar{\lambda}\omega}^dT_a.$$ By \cite{6}
or \cite{10}, we have
\begin{equation}
\varphi_{\bar{\lambda}\omega}^dT_a=\left\{
\begin{array}{cc}
q\varphi_{\bar{\lambda}\omega}^d& \text{if } i,i+1 \text { belong  to  the  same  row  of}\  t^{\bar{\lambda}}d,\\
\varphi_{\bar{\lambda}\omega}^{da}&\qquad\qquad \text{if  the  row  index  of}\  i\ \text{ in}\  t^{\bar{\lambda}}\  \text{is  less  than  that  of}\  i+1,\\
q\varphi_{\bar{\lambda}\omega}^{da}+(q-1)\varphi_{\bar{\lambda}}\varphi_{\bar{\lambda}\omega}^d&
\text{otherwise}.
\end{array}
\right.
\end{equation}
Then the proof is completed through checking  the formula above
case by case.
\end{proof}

By the definition in Remark \ref{redefinition}, we can show the following theorem on basis, which is the main result in this paper.
\begin{theorem}\label{mtsec2.1}
Suppose that $\lambda\in\Lambda_{n,r}^+(\textbf{m})$. Then the $q$-Schur module
$\mathcal {A}^\lambda$ is free as an $R$-module and
$\{\varphi^{1_A}_{\mu\lambda}\cdot z_\lambda|A\in \mathcal
{T}_\mu^{ss}(\lambda)\  and\
\mu\in\Lambda_{n,r}(\textbf{m})\}\subseteq\mathcal {A}^\lambda$ is a basis.

\end{theorem}

\begin{proof}

With the help of Theorem \ref{mtsec2.2}, it is enough to show that $\{\varphi^{1_A}_{\mu\lambda}z_\lambda|A\in \mathcal
{T}_\mu^{ss}(\lambda)\  and\
\mu\in\Lambda_{n,r}(\textbf{m})\}\subseteq\mathcal {A}^\lambda$ is $R$-linearly independent.
We calculate the action of the element
$\varphi_{\lambda\mu}^{1_A}\cdot z_\mu$ on the
unit of $\mathcal{H}_{n,r}$,
\begin{eqnarray*}
\varphi_{\lambda\mu}^{1_A}\cdot z_\mu(1)=\varphi^{1_A}_{\lambda\mu}\varphi^1_{\mu\omega}T_{w_\mu}y_{\mu'}(1)&=&\varphi^{1_A}_{\lambda\mu}(x_\mu)T_{w_\mu}y_{\mu'}\\
&=&(\sum\limits_{d\in\mathfrak{S}_{\bar{\lambda}}
1_A\mathfrak{S}_{\bar{\mu}}}T_d)\cdot
u^+_{[\mu]}T_{w_\mu}y_{\bar{\mu}'}u^-_{[\mu']}\qquad \qquad \text{by \cite{1}}\\
&=&(\sum\limits_{d\in\mathfrak{S}_{\bar{\lambda}}
1_A\mathfrak{S}_{\bar{\mu}}}T_d)\cdot T_{w_{(1)}\cdots
w_{(m)}}u^+_{[\mu]}T_{w_{[\mu]}}u^-_{[\mu']}y_{\bar{\mu}'}\\
&=&\varphi_{\bar{\lambda}\bar{\mu}}^{1_A}(x_{\bar{\mu}})\cdot
T_{w_{(1)}\cdots
w_{(r)}}v_{[\mu]}y_{\mu^{(r)}{}'\vee\cdots\vee\mu^{(1)}{}'}\qquad \qquad \text{by Lemma \ref{lemdu}}\\
&=&\varphi_{\bar{\lambda}\bar{\mu}}^{1_A}(x_{\bar{\mu}})
\cdot T_{w_{(1)}\cdots
w_{(r)}}\cdot y_{\mu^{(1)}{}'\vee\cdots\vee\mu^{(r)}{}'}\cdot v_{[\mu]}\qquad \qquad \text{by \cite{11}}\\
&=&\varphi_{\bar{\lambda}\bar{\mu}}^{1_A}(x_{\mu^{(1)}\vee\cdots\vee\mu^{(r)}}
T_{w_{(1)}\cdots
w_{(r)}} y_{\mu^{(1)}{}'\vee\cdots\vee\mu^{(r)}{}'})\cdot v_{[\mu]}\\
&=&\varphi_{\bar{\lambda}\bar{\mu}}^{1_A}\varphi_{\bar{\mu}\omega}^1\cdot
T_{w_{(1)}\cdots
w_{(r)}} y_{\mu^{(1)}{}'\vee\cdots\vee\mu^{(r)}{}'}(1)\cdot v_{[\mu]}\\
\end{eqnarray*}
Then, following from the calculation in \cite{6}, for
$A,B\in\mathcal {T}_{\bar{\lambda}}(\bar{\mu})$, we write $A\thicksim
B$ if $A$ and $B$ are \emph{row equivalent} (which has been defined in \cite{8}, i.e., if one tableau $A$ can be changed to $B$ by a sequence of elementary row permutations.). Thus,
$\mathfrak{S}_{\bar{\lambda}}1_A\mathfrak{S}_{\bar{\mu}}=\bigcup_{B\thicksim
A}\mathfrak{S}_{\bar{\lambda}}1_B$. In addition, if
$w\in\mathfrak{S}_n$, we denote by $\overline{w}$ the unique element
of $\mathfrak{S}_\lambda w\cap\mathscr{D}_\lambda$ for some
$\lambda\in\Lambda(n,r)$, i.e., the shortest element in $\mathfrak{S}_\lambda w$.
\begin{eqnarray*}
&&\varphi_{\bar{\lambda}\bar{\mu}}^{1_A}\varphi_{\bar{\mu}\omega}^1\cdot
T_{w_{(1)}\cdots w_{(r)}} y_{\mu^{(1)}{}'\vee\cdots\vee\mu^{(r)}{}'}\\
&=&(\sum\limits_{B\thicksim
A}\varphi_{\bar{\lambda}\omega}^{1_B}T_{w_{(1)}\cdots
w_{(r)}})y_{\mu^{(1)}{}'\vee\cdots\vee\mu^{(r)}{}'}\\
&=&(\sum\limits_{B\thicksim
A}\varphi^1_{\bar{\lambda}\omega}T_{1_B}T_{w_{(1)}\cdots
w_{(r)}})y_{\mu^{(1)}{}'\vee\cdots\vee\mu^{(r)}{}'}\\
&=&(\sum\limits_{B\thicksim
A}q^{K_B}\varphi^1_{\bar{\lambda}\omega}T_{\overline{1_B
w_{(1)}\cdots w_{(r)}}}+s_B)\cdot
y_{\mu^{(1)}{}'\vee\cdots\vee\mu^{(r)}{}'}\qquad \qquad \text{by \cite{6}}
\end{eqnarray*}
where $K_B$ is an integer and $s_B$ is a linear combination of terms
$\varphi^d_{\bar{\lambda}\omega}$ for which
$$\chi(t^{\bar{\lambda}}1_B,t^{\bar{\mu}})>\chi(t^{\bar{\lambda}}d,t^{\bar{\mu}}w_{(1)}\cdots
w_{(r)}).$$ Moreover,
$\chi(t^{\bar{\lambda}}1_A,t^{\bar{\mu}})>\chi(t^{\bar{\lambda}}1_B,t^{\bar{\mu}})=\chi(t^{\bar{\lambda}}\overline{1_Bw_{(1)}\cdots
w_{(r)}},t^{\bar{\mu}}1_Bw_{(1)}\cdots w_{(r)})$ if $B\thicksim A$
but $B\neq A$. Hence
\begin{eqnarray}\label{resultvarphi}
\varphi_{\bar{\lambda}\bar{\mu}}^{1_A}\varphi_{\bar{\mu}\omega}^1\cdot
T_{w_{(1)}\cdots w_{(r)}} y_{\mu^{(1)}{}'\vee\cdots\vee\mu^{(r)}{}'}
=(q^K\varphi_{\bar{\lambda}\omega}^1T_{\overline{1_Aw_{(1)}\cdots
w_{(r)}}}+s)\cdot y_{\mu^{(1)}{}'\vee\cdots\vee\mu^{(r)}{}'}
\end{eqnarray}
where $K$ is an integer and $s$ is a linear combination of terms
$\varphi_{\bar{\lambda}\omega}^d$ with
$$\chi(t^{\bar{\lambda}}1_A,t^{\bar{\mu}})>\chi(t^{\bar{\lambda}}d,t^{\bar{\mu}}w_{(1)}\cdots
w_{(r)}).$$

Now suppose that $\sum\limits_A
c_A\varphi_{\bar{\lambda}\bar{\mu}}^{1_A}\varphi_{\bar{\mu}\omega}^1\cdot
T_{w_{(1)}\cdots w_{(r)}} y_{\mu^{(1)}{}'\vee\cdots\vee\mu^{(r)}{}'}=0$,
where $c_A\in R$ and the sum is over $A\in
\mathcal{T}^{ss}_\lambda(\mu)$. Choose $D\in
\mathcal{T}^{ss}_\lambda(\mu)$ such that $c_A=0$ for all $A$ with
$\chi(t^{\bar{\lambda}}1_A,t^{\bar{\mu}})>\chi(t^{\bar{\lambda}}1_D,t^{\bar{\mu}})$.
If we can prove that $c_D=0$, it will follow that every
coefficient $c_A=0$, and then the proof is completed.

By (\ref{resultvarphi}), there exists an integer $K$ and
$s\in M^\lambda$ such that
\begin{eqnarray*}
\sum\limits_A
c_A\varphi_{\bar{\lambda}\bar{\mu}}^{1_A}\varphi_{\bar{\mu}\omega}^1\cdot
T_{w_{(1)}\cdots w_{(r)}}
y_{\mu^{(1)}{}'\vee\cdots\vee\mu^{(r)}{}'}=c_Dq^K\varphi^1_{\bar{\lambda}\omega}
T_{\overline{1_Dw_{(1)}\cdots
w_{(r)}}}y_{\mu^{(1)}{}'\vee\cdots\vee\mu^{(r)}{}'}+sy_{\mu^{(1)}{}'\vee\cdots\vee\mu^{(r)}{}'}
\end{eqnarray*}
where $s$ is a linear combination of terms
$\varphi_{\bar{\lambda}\omega}^d$($d\in
\mathscr{D}_{\bar{\lambda}}$) for which
\begin{eqnarray}\label{neq}
\chi(t^{\bar{\lambda}}d,t^{\bar{\mu}}w_{(1)}\cdots w_{(r)})\ngeq
\chi(t^{\bar{\lambda}}1_D,t^{\bar{\mu}}).
\end{eqnarray}

Now, suppose
\begin{eqnarray*}
c_Dq^K\varphi^1_{\bar{\lambda}\omega} T_{\overline{1_Dw_{(1)}\cdots
w_{(r)}}}y_{\mu^{(1)}{}'\vee\cdots\vee\mu^{(r)}{}'}+sy_{\mu^{(1)}{}'\vee\cdots\vee\mu^{(r)}{}'}=0
\end{eqnarray*}
and by Lemma \ref{lemmas2}, $\varphi^1_{\bar{\lambda}\omega}
T_{\overline{1_Dw_{(1)}\cdots
w_{(r)}}}y_{\mu^{(1)}{}'\vee\cdots\vee\mu^{(r)}{}'}$ is the linear
combination of terms $\varphi_{\bar{\lambda}\omega}^d$ ($d\in
\mathscr{D}_{\bar{\lambda}}$) for which
$\chi(t^{\bar{\lambda}}d,t^{\bar{\mu}}w_{(1)}\cdots
w_{(r)})=\chi(t^{\bar{\lambda}}\overline{1_Dw_{(1)}\cdots
w_{(r)}},t^{\bar{\mu}}w_{(1)}\cdots
w_{(r)})=\chi(t^{\bar{\lambda}}1_D,t^{\bar{\mu}})$, while
$sy_{\mu^{(1)}{}'\vee\cdots\vee\mu^{(r)}{}'}$ is a linear combination of
terms $\varphi_{\bar{\lambda}\omega}^1$($d\in \mathscr{D}_\lambda$)
for which
$\chi(t^{\bar{\lambda}},t^{\bar{\mu}})\neq\chi(t^{\bar{\lambda}}1_D,t^{\bar{\mu}})$
by (\ref{neq}). Therefore,
$$c_Dq^K\varphi^1_{{\bar{\lambda}}\omega}T_{\overline{1_Dw_{(1)}\cdots
w_{(r)}}}y_{\mu^{(1)}{}'\vee\cdots\vee\mu^{(r)}{}'}=0.$$ But
$\varphi^1_{{\bar{\lambda}}\omega}T_{\overline{1_Dw_{(1)}\cdots
w_{(r)}}}y_{\mu^{(1)}{}'\vee\cdots\vee\mu^{(r)}{}'}\neq 0$, since the
numbers strictly increase down the columns for every component of $D$. Therefore,
$c_D=0$, as we claimed.

Now, we have already known that the elements
$\varphi_{\bar{\lambda}\bar{\mu}}^{1_A}\varphi_{\bar{\mu}\omega}^1\cdot
T_{w_{(1)}\cdots w_{(r)}} y_{\mu^{(1)}{}'\vee\cdots\vee\mu^{(r)}{}'}$ is
linearly independent. It implies that $\varphi^{1_A}_{\lambda\mu}\varphi^1_{\mu\omega}T_{w_\mu}y_{\mu'}=\varphi_{\bar{\lambda}\bar{\mu}}^{1_A}\varphi_{\bar{\mu}\omega}^1\cdot
T_{w_{(1)}\cdots
w_{(r)}} y_{\mu^{(1)}{}'\vee\cdots\vee\mu^{(r)}{}'}\cdot v_{[\mu]}$ are $R$-linearly independent, since by Lemma \ref{lemdu} it is
trivial that $a\cdot v_{[\mu]}=0$ if and only if $a=0$ for any
$a\in\mathcal {H}(\mathfrak{S}_r)$.
\end{proof}

\section{\bf Application to  the Branch rule}
In this section, by using this embedding and restriction functors arised in \cite{23},
we give a new proof of the Branch rule in a cyclotomic $q$-Schur algebra of rank $n$ to one of rank $n+1$.

From now on, throughout this paper, we argue under the following setting:
\begin{eqnarray*}
&&\textbf{m}=(m_1,\cdots,m_r)\text{ such that }m_k\geq n+1\text{ for all }k=1,\cdots,r,\\
&&\textbf{m}'=(m_1,\cdots,m_{r-1},m_r-1),\\
&&\mathscr{S}_{n+1,r}={}_R\mathscr{S}_{n+1,r}(\Lambda_{n+1,r}(\textbf{m})),\\
&&\mathscr{S}_{n,r}={}_R\mathscr{S}_{n,r}(\Lambda_{n,r}(\textbf{m}')).
\end{eqnarray*}
We will omit the subscript $R$ when there is no risk to confuse.

Define the injective map
\begin{eqnarray*}
\gamma:\Lambda_{n,r}(\textbf{m}')\rightarrow \Lambda_{n+1,r}(\textbf{m}), \qquad (\lambda^{(1)},\cdots,\lambda^{(r-1)},\lambda^{(r)})\mapsto (\lambda^{(1)},\cdots,\lambda^{(r-1)},\widehat{\lambda}^{(r)}),
\end{eqnarray*}
where $\widehat{\lambda}^{(r)}=(\lambda_1^{(1)},\cdots,\lambda^{(r)}_{m_r-1},1)$. Put $\Lambda^{\gamma}_{n+1,r}(\textbf{m})= \text{Im} \gamma$,  we have
\begin{eqnarray*}
\Lambda_{n+1,r}^\gamma(\textbf{m})=\{\mu=(\mu^{(1)},\cdots,\mu^{(r)})\in \Lambda_{n+1,r}(\textbf{m})|\mu^{(r)}_{m_r}=1\},
\end{eqnarray*}
where it is defined that $\mu^{(i)}=(\mu^{(i)}_1,\cdots,\mu^{(r)}_{m_i})\in \mathbb{Z}^{m_i}_{>0}$ for $1\leq i\leq r$.

For $\lambda\in \Lambda_{n+1,r}^+$, and $T\in \mathcal{T}^{ss}_\Lambda(\lambda)$, let $T\setminus (n+1)$ be the standard tableau obtained by removing the node
$x$ such that $T(x)=n+1$, and denote the shape of $T\setminus(n+1)$ by $\text{Shape}(T\setminus(n+1))$. Note that $x$ here is a removable node of $\lambda$, and that $\text{Shape}(T\setminus(n+1))=\lambda\setminus x$.

\begin{proposition}\cite{23}(\textbf{Wada inclusion})\label{wadainclusion}
There exists an algebra homomorphism $\iota:\mathscr{S}_{n,r}\rightarrow\mathscr{S}_{n+1,r}$ such that
\begin{eqnarray}
E_{(i,k)}^{(l)}\mapsto E_{(i,k)}^{(l)}\xi,\qquad F_{(i,k)}^{(l)}\mapsto F_{(i,k)}^{(l)}\xi,\qquad 1_\lambda\mapsto 1_{\gamma(\lambda)}
\end{eqnarray}
for $(i,k)\in \Gamma'(\textbf{m}')$, $l\geq1$, $\lambda\in \Lambda_{n,r}(\textbf{m}')$, where $\xi=\sum_{\lambda\in \Lambda^{\gamma}_{n+1,r}(\textbf{m})}1_\lambda$ is an idempotent of
$\mathscr{S}_{n+1,r}$. In particular, we have that $\iota(1_{\mathscr{S}_{n,r}})=\xi$, and that $\iota(\mathscr{S}_{n,r})\subsetneq\xi\mathscr{S}_{n+1,r}\xi$, where $1_{\mathscr{S}_{n,r}}$
is the unit element of $\mathscr{S}_{n,r}$. Moreover, $\iota $ is injective.
\end{proposition}

 Define a restriction functor $\text{Res}_n^{n+1}:\mathscr{S}_{n+1,r}$-mod$\rightarrow$ $\mathscr{S}_{n,r}$-mod by
\begin{eqnarray*}
\text{Res}_{n}^{n+1}=\text{Hom}_{\mathscr{S}_{n+1,r}}(\mathscr{S}_{n+1,r}\xi,-)\cong \xi\mathscr{S}_{n+1,r}\otimes_{\mathscr{S}_{n+1,r}}-.
\end{eqnarray*}

Recall that, for $\lambda\in \Lambda^+_{n+1,r}$, the $q$-Schur module $\mathcal{A}^\lambda$ of $\mathscr{S}_{n+1,r}$ has an $R$-free basis $\{\varphi_{\mu\lambda}^{1_A}z_\lambda|A\in \mathcal{T}^{ss}_\mu(\lambda),\mu\in \Lambda_{n+1,r}(\textbf{m})\}$. From the definition, we have that
\begin{eqnarray*}
\text{Res}_n^{n+1}(\mathcal{A}^\lambda)=\xi\mathcal{A}^\lambda.
\end{eqnarray*}
Thus, $\text{Res}_n^{n+1}(\mathcal{A}^\lambda)$ has an $R$-free basis $\{\varphi_{\mu\lambda}^{1_A}z_\lambda|A\in \mathcal{T}^{ss}_\mu(\lambda),\mu\in \Lambda_{n+1,r}^\gamma(\textbf{m})\}$.

For a partition $\lambda=(\lambda_1,\cdots,\lambda_m)$ of $n$, we identify the boxes in the Young diagram $\mathcal {N}(\lambda)$ with its position coordinates. Thus,
\begin{eqnarray*}
\mathcal{N}(\lambda)=\{(i,j)\in \mathbb{Z}^+\times\mathbb{Z}^+|j\leq \lambda_i\}.
\end{eqnarray*}
The elements of $\mathcal{N}(\lambda)$ will be called \emph{nodes}. A node of the form $(i,\lambda_i)$ (resp. $(i,\lambda_{i}+1)$) is called \emph{removable} (resp. \emph{addable})
if $i=m$ or $\lambda_i>\lambda_{i+1}$ for $i\neq m$ (resp. $(i,\lambda_i)=(0,1)$ for $\lambda_1=\cdots=\lambda_m=1$ or $i=1$ or $\lambda_{i-1}>\lambda_i$ if $i\neq 1$).

 Let $\lambda=(\lambda^{(1)},\cdots,\lambda^{(r)})$ be an $r$-partition. Then its $\mathcal{N}(\lambda)$ is the union of $\mathcal{N}(\lambda^{(k)})$, $1\leq k\leq r$.
i.e.,  a set of nodes
\begin{eqnarray*}
\mathcal{N}(\lambda)=\{(i,j,k)|i,j\in \mathbb{Z}^+,j\leq \lambda^{(k)}_i,\  1\leq k\leq m\}.
\end{eqnarray*}
A node of $\mathcal{N}(\lambda)$ is said to be {\em removable} (resp. {\em addable}) if it is a removable (resp. addable) node of $\mathcal{N}(\lambda^{(k)})$ for some $k$. Denote by $\mathcal {R}_\lambda$ the set of all removable nodes of $\mathcal{N}(\lambda)$. Then $N=\#\mathcal{R}_\lambda=\sum_{i=1}^r\#\mathcal{R}_{\lambda^{(i)}}$.

A partial ordering $``\succ"$ on $\mathcal{R}_\lambda$ will be fixed from top to bottom and from left to right, that is, it satisfies that
 $$(i,j,k)\succ (i',j',k')\text{ if }k<k'\text{, or if }k=k'\text{ and }i<i'.$$
Then, we have $\mathcal{R}_\lambda=\{\mathfrak{n}_1,\cdots,\mathfrak{n}_N\}$, with the property that $\mathfrak{n}_i\succ\mathfrak{n}_j$ for $i>j$. Let $j_{\mathfrak{n}}$, $\mathfrak{n}\in \mathcal{R}_\lambda$, be the number at the node $\mathfrak{n}$ in $\mathfrak{t}_\lambda$. For example, for $\lambda=\big((31),(22),(1)\big)$, $\mathcal{R}_\lambda=\{(1,3,1), (2,1,1),(1,1,3)\}$.

Also, we define a partial order $\succeq$ on $\mathbb{Z}_{>0}\times\{1,\ldots,r\}$ by
\begin{eqnarray*}
(i,k)\succ (i',k') \text{ if }(i,1,k)\succ(i',1,k').
\end{eqnarray*}

\begin{proposition}\label{multi}\label{proposition}
Let $\lambda\in \Lambda_{n+1,r}^+$, $\mu\in \Lambda^\gamma_{n+1,r}(\textbf{m})$, $A\in \mathcal{T}^{ss}_\mu(\lambda)$. For $(i,k)\in \Gamma'(\textbf{m}')$,
we have the following
\begin{eqnarray}
&&E_{(i,k)}\cdot \varphi_{\mu\lambda}^{1_A}z_\lambda=\sum\limits_{\substack{B\in \mathcal{T}^{ss}_{\mu+\alpha_{(i,k)}}(\lambda)\\ \text{shape}(B\setminus(m_r,r))\trianglerighteq \text{shape}(A\setminus(m_r,r))}}r_B\varphi_{\mu+\alpha_{(i,k)},\lambda}^{1_B}z_\lambda                   \text{    }(r_B\in R);\\
&&F_{(i,k)}\cdot \varphi_{\mu\lambda}^{1_A}z_\lambda=\sum\limits_{\substack{B\in \mathcal{T}^{ss}_{\mu-\alpha_{(i,k)}}(\lambda)\\ \text{shape}(B\setminus(m_r,r))\trianglerighteq \text{shape}(A\setminus(m_r,r))}}r_B\varphi_{\mu-\alpha_{(i,k)},\lambda}^{1_B}z_\lambda                   \text{    }(r_B\in R).
\end{eqnarray}

\end{proposition}

\begin{proof}
Following from  the notations of (5.8), (5.9) in \cite{1}, one shows that $\varphi^{1_A}_{\mu\lambda}=\Psi_{AT^{\lambda}}$. On the other hand, by a general
theory of cellular algebras together within Wada's paper \cite{23} (Proposition 3.3), it implies for $(i,k)\in \Gamma'(\textbf{m}')$,
\begin{eqnarray}\label{equ}
E_{(i,k)}\cdot \varphi_{\mu\lambda}^{1_A}\equiv\sum\limits_{\substack{B\in \mathcal{T}^{ss}_{\mu+\alpha_{(i,k)}}(\lambda)\\ \text{shape}(B\setminus(m_r,r))\trianglerighteq \text{shape}(A\setminus(m_r,r))}}r_B\varphi_{\mu+\alpha_{(i,k)},\lambda}^{1_B}\qquad \text{mod}\   \mathscr{S}_{n+1,r}^{\vartriangleright\lambda},
\end{eqnarray}
where $r_B\in R$.

By definitions, $z_\lambda:=\varphi_{\lambda\omega}^1T_w y_{\lambda'}$ and $\mathscr{S}_{n+1,r}^{\vartriangleright\lambda}$ is linearly generated by $\Psi_{ST}$ for $S,T\in \mathcal{T}_\Lambda(\nu)$ with $\nu\vartriangleright\lambda$. It follows that $\mathscr{S}_{n+1,r}^{\vartriangleright\lambda}\cdot z_\lambda=0$. On the other hand,  suppose that there exists some $S,T\in \mathcal{T}^{ss}_\Lambda(\nu)$ such that $\Psi_{ST}z_\lambda\neq 0$, which  means $\lambda=\nu$ due to the proof
of Theorem \ref{mtsec2.2}. This consequence is contradict to the fact $\nu\vartriangleright\lambda$. Finally, we reach the consequence of the first statement after multiple the element $z_\lambda$ on the two sides of (\ref{equ}).

The case for $F_{(i,k)}$ with $(i,k)\in \Gamma'(\textbf{m}')$ can be proved similarly with the above proof in the  case for $E_{(i,k)}$.
\end{proof}

By Theorem \ref{mtsec2.1},
let ${}_R M_i$ be an $R$-submodule of $\text{Res}_n^{n+1}(\mathcal{A}^\lambda)$ spanned by
\begin{eqnarray*}
\{\varphi_{\mu\lambda}^{1_A}z_\lambda|A\in \mathcal{T}_\Lambda^\gamma(\lambda)\cap\mathcal{T}^{ss}_\Lambda(\lambda)\text{ such that }A(\mathfrak{n}_j)=(m_r,r)\text{ for some }j\geq i\},
\end{eqnarray*}
where we put $\mathcal{T}^\gamma_\Lambda(\lambda):=\bigcup_{\mu\in \Lambda^\gamma_{n+1,r}(\textbf{m})}\mathcal{T}_\mu(\lambda)$. When there is no confusion about $R$, we also denote ${}_R M_i$ as $M_i$.
Then we have a filtration of $R$-modules
$$\text{Res}^{n+1}_n(\mathcal{A}^\lambda)=M_1\supset M_2\supset\cdots \supset M_k \supset M_{k+1}=0.$$

For $\lambda\in \Lambda^+_{n+1,r}$ and a removable node $x$ of $\lambda$,  define the semi-standard tableau $T^\lambda_x\in \mathcal{T}^{ss}_\Lambda(\lambda)$ by
\begin{eqnarray}
T^\lambda_x(a,b,c)=\left\{
\begin{array}{cc}
(a.c)\qquad\quad \text{if}~ (a,b,c)\neq x,\\
(m_r,r)\qquad \text{if}~ (a,b,c)=x.
\end{array}
\right.
\end{eqnarray}
We see that $T^\lambda_x\in \mathcal{T}_\Lambda^\gamma(\lambda)\cap\mathcal{T}^{ss}_\Lambda(\lambda) $, and $T^\lambda_x\setminus(m_r,r)=T^{\lambda\setminus x}$,
where the tableau $T^{\lambda\setminus x}$ notes the unique element in the set $\mathcal{T}_{\lambda\setminus x}^{ss}(\lambda\setminus x)$.

From the definition, $M_i/M_{i+1}$ has an $R$-free basis
\begin{eqnarray*}
\{\varphi_{\gamma(\mu)\lambda}^{1_A} z_\lambda+M_{i+1}|A\in \mathcal{T}_\Lambda^\gamma(\lambda)\cap\mathcal{T}^{ss}_\Lambda(\lambda)\text{ such that } A(\mathfrak{n}_i)=(m_r,r)\text{ and } \mu\in \Lambda_{n,r}(\textbf{m})\}.
\end{eqnarray*}
For $A\in \mathcal{T}_\Lambda^\gamma(\lambda)\cap\mathcal{T}^{ss}_\Lambda(\lambda)$ such that $A(\mathfrak{n}_i)=(m_r,r)$, we have $\text{Shape}(A\setminus(m_r,r))=\lambda\setminus \mathfrak{n}_i$ by the definition. Note that $\lambda\setminus\mathfrak{n}_j\rhd\lambda\setminus\mathfrak{n}_i$ if and only if $\mathfrak{n}_j\prec\mathfrak{n}_i$ (i.e., $j>i$). Then, by Proposition \ref{multi}, we see that $\{M_i\}$ is a filtration of $\mathscr{S}_{n,r}$-modules.

Now, we use the main result in Section 3 to give a new proof of the Branch rule of Weyl modules in \cite{23}.

\begin{theorem}\cite{23}\; Assume that $R$ is a field.
For any $\lambda\in \Lambda^+_{n+1,r}(\textbf{m})$, let $\mathfrak{n}_1,\cdots,\mathfrak{n}_k$ be the removable nodes of $\mathcal{N}(\lambda)$ counted from top to bottom, and define
$M_t$ as above for $1\leq t\leq k$. Then, we have a filtration of $\mathscr{S}_{n,1}$-submodule for $\mathcal{A}^\lambda$:
$$0=M_{k+1}\subset M_k\subset\cdots\subset M_1=\mathcal{A}^\lambda$$
with the sections of Weyl modules (or $q$-Schur modules): $M_t/M_{t-1}\cong W^{\lambda\setminus\mathfrak{n}_t}$.

\end{theorem}

\begin{proof}

First of all we set $\widehat{\mu}:=\gamma(\mu)$, and consider the weight decomposition of $\mathscr{S}_{n,r}$-module $M_{i}/M_{i+1}=\bigoplus\limits_{\mu\in \Lambda_{n,r}(\textbf{m})} {}_\mu(M_i/M_{i+1})=\bigoplus\limits_{\mu\in \Lambda_{n,r}(\textbf{m})} 1_{\mu}\cdot M_i/M_{i+1}=\bigoplus \limits_{\mu\in \Lambda_{n,r}(\textbf{m})} 1_{\widehat{\mu}}(M_{i}/M_{i+1})$,
where $1_{\widehat{\mu}}(M_{i}/M_{i+1})$ is generated by
$$\{\varphi_{\widehat{\mu}\lambda}^{1_A} z_\lambda+M_{i+1}|A\in \mathcal{T}_{\Lambda}^\gamma(\lambda)\cap\mathcal{T}^{ss}_{\Lambda}(\lambda)\text{ such that } A(\mathfrak{n}_i)=(m_r,r)\}.$$
Since $A\setminus(m_r,r)\in \mathcal{T}^{ss}_\mu(\lambda\setminus\mathfrak{n}_i)$, we can find that ${}_\mu(M_i/M_{i+1})\neq0$ only if $\lambda\unrhd\widehat{\mu}$, which implies that $\lambda\setminus\mathfrak{n}_i\unrhd\mu$.

Let $\mathfrak{n}_i=(a,b,c)$. Note that $E_{(j,l)}\cdot \varphi_{\widehat{\mu}\lambda}^{1_A}z_\lambda$ is a linear combination of $\{\varphi_{\widehat{\mu}+\alpha_{(j,l)},\lambda}^{1_B}z_\lambda|B\in \mathcal{T}^{ss}_{\widehat{\mu}+\alpha_{(j,l)}}(\lambda)\}$ and that $\mathcal{T}^{ss}_{\widehat{\mu}+\alpha_{(j,l)}}(\lambda)=\emptyset$ unless $\lambda\unrhd \widehat{\mu}+\alpha_{(j,l)}$.

We have $T_{\mathfrak{n}_i}^\lambda
\in \mathcal{T}^{ss}_\tau(\lambda)$ in the case of $\tau:=\widehat{\lambda\setminus\mathfrak{n}_i}$, i.e., $\tau=\lambda-(\alpha_{(a,c)}+\alpha_{(a+1,c)}+\cdots+\alpha_{(m_r-1,r)})$.

If $(j,l)\succ(a,c)$, we have $E_{(j,l)}\cdot  \varphi_{\tau\lambda}^{1_A}z_\lambda=0$ since $\lambda\ntrianglerighteq \tau+\alpha_{(j,l)} \text{ for any }A\in \mathcal{T}^{ss}_\tau(\lambda)$.

If $(j,l)\preceq(a,c)$, for any $S\in \mathcal{T}^{ss}_{\tau+\alpha_{(j,l)}}(\lambda)$ together with the definition of semi-standard tableaux, we can easily
check that $S\big((a',b',c')\big)\succeq(j,l)$ for any $(a',b',c')\in \lambda$ satisfying $(a',c')\succeq (j,l)$. This implies that
\begin{eqnarray}\label{used}
|S\setminus(m_r,r)|\neq|\lambda\setminus \mathfrak{n}_i|\text{ for any }S\in \mathcal{T}^{ss}_{\tau+\alpha_{(j,l)}}(\lambda),
\end{eqnarray}
since $(a,c)\succeq(j,l)$ and $T^\lambda_{\mathfrak{n}_i}\big((a,b,c)\big)=(m_r,r)\preceq (j,l)$. From now on, we note the tableau $T_{\mathfrak{n}_i}^\lambda$ as $X$.

Thus, Proposition \ref{proposition} together with (\ref{used}) implies that
\begin{eqnarray*}
E_{(j,l)}\cdot\varphi_{\tau\lambda}^{1_{X}}\cdot z_\lambda=0\in M_{i+1}\text{ for any }(j,l)\in \Gamma'(\textbf{m}').
\end{eqnarray*}
Thus, $\varphi_{\tau\lambda}^{1_{X}}\cdot z_\lambda+M_{i+1}$ is a highest weight vector of weight $\lambda\setminus\mathfrak{n}_i$ of $\mathscr{S}_{n,r}$-module in sense of \cite{22}. Moreover, since the Weyl modules are simple modules in category of ${}_\mathcal{K}\mathscr{S}_{n,r}$-modules,  due to the universality of Weyl modules in \cite{22}, we
have an ${}_\mathcal{K}\mathscr{S}_{n,r}$-isomorphism:
\begin{eqnarray}
\theta^{\lambda\setminus\mathfrak{n}_i}_\mathcal{K}:\qquad {}_\mathcal{K}\mathcal{A}^{\lambda\setminus\mathfrak{n}_i}\rightarrow{}_\mathcal{K}\mathscr{S}_{n,r}\cdot(\varphi_{\tau\lambda}^{1_{X}}\cdot z_\lambda)+{}_\mathcal{K}M_{i+1}.
\end{eqnarray}
Note that $\theta^{\lambda\setminus\mathfrak{n}_i}_\mathcal{K}$ is determined by $\theta^{\lambda\setminus\mathfrak{n}_i}_\mathcal{K}(\varphi_{{\lambda\setminus\mathfrak{n}_i}{\lambda\setminus\mathfrak{n}_i}}^1\cdot z_{\lambda\setminus\mathfrak{n}_i})=\varphi_{\tau\lambda}^{1_{X}}\cdot z_\lambda+{}_\mathcal{K}M_{i+1}$. We see that $\theta^{\lambda\setminus\mathfrak{n}_i}_\mathcal{A}$ is a restriction of $\theta^{\lambda\setminus\mathfrak{n}_i}_\mathcal{K}$ which assigns the submodule
$_\mathcal{A}\mathcal{A}^{\lambda\setminus\mathfrak{n}_i}$ onto the submodule ${}_\mathcal{A}\mathscr{S}_{n,r}\cdot(\varphi_{\tau\lambda}^{1_{X}}\cdot z_\lambda)+{}_\mathcal{A}M_{i+1}$. Then, we find that $\theta_\mathcal{A}^{\lambda\setminus\mathfrak{n}_i}$ is an isomorphism of ${}_\mathcal{A}\mathscr{S}_{n,r}$-modules. Furthermore, by the argument of specialization to any arbitrary commutative ring, it follows that $\theta^{\lambda\setminus\mathfrak{n}_i}_{R}:=\theta^{\lambda\setminus\mathfrak{n}_i}_\mathcal{A}\otimes_\mathcal{A}R$ is an isomorphism for the algebra ${}_R\mathscr{S}_{n,r}$.

$R$ is assumed to be a field. Since $W^{\lambda\setminus \mathfrak{n}_i}\cong\mathcal{A}^{\lambda\setminus\mathfrak{n}_i}\cong {}_R\mathscr{S}_{n,r}\cdot(\varphi_{\tau\lambda}^{1_{X}}\cdot z_\lambda)+{}_R M_{i+1}$, which is a ${}_R\mathscr{S}_{n,r}$-submodule of $M_i/M_{i+1}$, we finally reach the consequence by comparing the dimensions of $\mathcal{A}^{\lambda\setminus\mathfrak{n}_i}$ and $M_i/M_{i+1}$.
\end{proof}

\noindent
{\bf Acknowledgements: }{\em The authors thank the support from the projects of the National Natural Science Foundation of China (No.11271318 and No.11171296) and the Specialized Research Fund for the Doctoral Program of Higher Education of China (No.20110101110010).  }


\bigskip
\begin{thebibliography}{99}
\bibitem{3}H. Can, Representations of the Generalized Symmetric Groups. Beitr$\ddot{a}$ge Alg. Geo. \textbf{1996}, 37, 289-307.
\bibitem{6}R. Dipper, G. James, $q$-Tensor space and
$q$-Weyl Modules. Transactions of the American Mathematical Society.
Vol. 327, No. 1 (Sep., 1991), Pages 251-282.
\bibitem{8}R. Dipper, G. James, A. Mathas, Cyclotomic $q$-Schur algebras, Math. Zeit., 229 (1999),
385-416.
\bibitem{10}R. Dipper, G. James, Representations of Hecke algebras
of general linear groups, Proc. L.M.S (3), \textbf{52} (1986),
20-50.
\bibitem{14} R. Dipper, G. James, Representations of the Hecke Algebras of Type $B_n$. J. Algebra \textbf{1992}, \emph{146}, 454-481.
\bibitem{11}J. Du, H.B. Rui, Ariki-Koike Algebras with Semi-simple
Bottoms. Math. Zeit. \textbf{2000}, 204, 807-835.
\bibitem{1}J. Du, H.B. Rui, Borel Type Subalgebras of the
$q$-Schur$^{m}$. Journal of Algebra, Volume 213, Issue 2, 15 March
1999, Pages 567-595
\bibitem{9}J. Du, H.B. Rui, Specht modules for Ariki-Koike algebras, Comm. Algebra 29 (2001) 4710-4719.
\bibitem{4}J. Du, B. Parshall and J.-p. Wang, Two-parameter quantum linear groups and the hyperbolic
invariance of q-Schur algebras, J. London Math. Soc. 44 (1991),
420-436.
\bibitem{15}W. Fulton, J. Harris, Representation Theory: A First Course, Springer-Verlag, 1991.
\bibitem{13}J. Graham, G. Lehrer, Cellular Algebras. Invent. Math.
1996, 126, 1-34.
\bibitem{20} J.E. Humphreys, Reflection Groups and Coxeter Groups. Cambridge: Cambridge
University Press,  1990
\bibitem{19}G.D. James,  A. Kerber, the representation theory of the symmetric group, \textbf{16}, Encyclopedia of Mathematics, Addison-Wesley, Massachusetts, 1981.
\bibitem{18}M. Jimbo, A q-analogue of U(gl(N + 1)), Hecke algebra and the Yang-Baxter equation,
Lett. Math. Phys. 11 (1986), 247-252.
\bibitem{21} T. Jost, Morita Equivalence for Blocks of Hecke Algebras of Symmetric Groups, Journal of Algebra, Volume 194, 201-223 (1997).
\bibitem{7}A. Mathas, Iwahori-Hecke Algebras and Schur Algebras of the Symmetric Group.
Univ. Lecture Ser., vol. 15 Amer. Math. Soc. (1999).
\bibitem{12}A. Mathas, The representation theory of the Ariki-Koike and cyclotomic q-Schur algebras, pp. 261-320, Adv. Stud. Pure Math., 40, Math. Soc. Japan, 2004.
\bibitem{16} A. Mathas, Seminormal forms and Gram determinants for cellular algebras, J. Reine Angew. Math. \textbf{619} (2008), 141-173.
\bibitem{17}A. Mathas, Tilting modules for cyclotomic Schur algebras, J. Reine Angew. Math., \textbf{562} (2003), 137-169.
\bibitem{23} K. Wada, Induction and Restriction Functors for Cyclotomic $q$-Schur Algebras, (2012) arXiv:1112.6068.
\bibitem{22} K. Wada, Presenting cyclotomic $q$-Schur algebras, \emph{Nagoya Math. J.} \textbf{201} (2011), 45-116.
\end{thebibliography}
\end{document}